\input ./style/arxiv-vmsta.cfg
\documentclass[numbers,compress,v1.0.1]{vmsta}
\usepackage{vtexbibtags}

\volume{7}
\issue{1}
\pubyear{2020}
\firstpage{79}
\lastpage{96}
\aid{VMSTA146}
\doi{10.15559/19-VMSTA146}
\articletype{research-article}

\startlocaldefs

\theoremstyle{definition}
\newtheorem{defi}{Definition}[section]
\newtheorem{ex}{Example}[section]
\newtheorem{rem}{Remark}[section]

\theoremstyle{plain}
\newtheorem{thm}{Theorem}[section]
\newtheorem{cor}{Corollary}[section]
\newtheorem{lmm}{Lemma}[section]

\numberwithin{equation}{section}

\allowdisplaybreaks
\ifdefined\HCode 
\def\index#1{}
\else 
\fi
\endlocaldefs

\begin{document}

\begin{frontmatter}
\pretitle{Research Article}

\title{Estimates for distribution of suprema of solutions to higher-order partial differential equations with random initial conditions}

\author[a]{\inits{Yu.}\fnms{Yuriy}~\snm{Kozachenko}\ead[label=e1]{ykoz@ukr.net}}
\author[b]{\inits{E.}\fnms{Enzo}~\snm{Orsingher}\ead[label=e2]{enzo.orsingher@uniroma1.it}}
\author[a]{\inits{L.}\fnms{Lyudmyla}~\snm{Sakhno}\ead[label=e3]{lms@univ.kiev.ua}}
\author[a]{\inits{O.}\fnms{Olga}~\snm{Vasylyk}\thanksref{cor1}\ead[label=e4]{ovasylyk@univ.kiev.ua}\orcid{0000-0002-0880-3751}}
\thankstext[type=corresp,id=cor1]{Corresponding author.}
\address[a]{Department of Probability Theory, Statistics and Actuarial Mathematics, \institution{Taras~Shevchenko National University of Kyiv}, 64 Volodymyrska str.,\break  01601, Kyiv, \cny{Ukraine}}
\address[b]{Department of Statistical Sciences, \institution{Sapienza University of Rome},\break  P.le Aldo Moro, 5, 00185, Rome, \cny{Italy}}



\markboth{Yu. Kozachenko et al.}{Estimates for distribution of suprema of solutions to higher-order partial differential equations}

\begin{abstract}
In the paper we consider higher-order partial differential equations
from the class of linear dispersive equations. We investigate
solutions to these equations subject to random initial conditions given
by harmonizable $\varphi $-sub-Gaussian processes. The main results are
the bounds for the distributions of the suprema for solutions. We
present the examples of processes for which the assumptions of the
general result are verified and bounds are written in the explicit form.
The main result is also specified for the case of Gaussian initial
condition.
\end{abstract}

\received{\sday{25} \smonth{7} \syear{2019}}
\revised{\sday{10} \smonth{10} \syear{2019}}
\accepted{\sday{16} \smonth{11} \syear{2019}}
\publishedonline{\sday{17} \smonth{12} \syear{2019}}

\end{frontmatter}

\section{Introduction}%
\label{intro}

Numerous recent studies are concerned with evolution equations of the
form
\begin{equation*}
\label{1intr}
\frac{\partial u}{\partial t}+\sum _{j=1}^{l}\frac{\partial ^{2j+1} u}{
\partial x^{2j+1}}+u^{k}\frac{\partial u}{\partial x}=0,
\quad
l,k\in \mathbb{N},
\end{equation*}
which are dispersive equations of order $2l+1$ with a convective term
$u^{k}\frac{\partial u}{\partial x}$; equations with coefficients of
more general form and different kinds of nonlinearity are also the subject
of active research.

The most celebrated equation of this class is the Korteweg--De Vries
(KdV) equation
%
\begin{equation}
\label{2intr}
\frac{\partial u}{\partial t}=-\frac{\partial ^{3} u}{\partial x^{3}}+u\frac{
\partial u}{\partial x},
\end{equation}
which describes the evolution of small amplitude long waves in fluids
and other media.

The Kawahara equation with dispersive terms of the third and fifth
orders
%
\begin{equation}
\label{3intr}
\frac{\partial u}{\partial t}+u\frac{\partial u}{\partial x}+\alpha \frac{
\partial ^{3} u}{\partial x^{3}}+\beta \frac{\partial ^{5} u}{\partial
x^{5}}=0
\end{equation}
is used to model various dispersive phenomena such as plasma waves,
capilarity-gravity water waves, etc. in situations when the cubic
dispersive term\index{cubic dispersive term} is weak or not sufficient. In the most recent research,
generalisations of Equations \eqref{2intr}, \eqref{3intr} have been
suggested and treated.

In the physical and mathematical literature the existence, uniqueness
and analytic properties of solutions to the initial value problem have
been intensively investigated for various linear and nonlinear
dispersive equations.\index{nonlinear dispersive equations} Boundary value problems for such equations were
also considered. We refer, for example, to the comprehensive study
undertaken in the book by Tao \cite{Tao}, among many other books
and papers on the topic.

One should note the importance of the study of constant coefficient
linear dispersive equations for its own sake and also because this
provides prerequisites for the theory of nonlinear dispersive equations,\index{nonlinear dispersive equations}
since the latter are often obtained by perturbation of the linear theory
(\cite{Tao}). Developing the theory of linear equations is also
essential for describing those evolution phenomena where the linear
effects compensate or dominate nonlinear ones. In such situations, one
can expect that the nonlinear solutions display almost the same behavior
as the linear ones.

In the probabilistic literature significant attention has been paid to
the equations of the form
%
\begin{equation}
\label{4intr}
\frac{\partial u}{\partial t}=k_{m}\frac{\partial ^{m}u}{\partial x^{m}},
\quad
x\in R,\ t>0,\ m\ge 2.
\end{equation}

The investigation of fundamental solutions to the equation
\eqref{4intr} can be traced back to works by Bern\v{s}tein and L\'{e}vy.
Such solutions are sign-varying and, based on them, the so-called
pseudoprocesses have been introduced and extensively investigated in the
literature. Note, that \xch{Equations}{equations}~(\ref{4intr}) of even and odd order
possess solutions of different structure and behaviour (for example, see
\cite{Orsingher2012}). We refer to
\cite{KozOrsSakhVas_JOSS2018}, where a review of the recent results on
this topic and additional literature are presented.

We note that in the probabilistic literature equations of the form
\eqref{4intr} and their generalizations are often called higher-order
heat-type equations.

Equations of the form \eqref{4intr} subject to random\index{random initial conditions} initial conditions
were studied in \cite{BKLO}, namely, the asymptotic behavior was
analysed for the rescaled solution to the Airy equation with random\index{random initial conditions}
initial conditions given by weakly dependent stationary processes.\index{stationary processes}

More general odd-order equations of the form
%
\begin{equation}
\label{1.8}
\frac{\partial u}{\partial t}=\sum _{k=1}^{N}a_{k}\frac{\partial ^{2k+1}u}{
\partial x^{2k+1}}\;,
\quad
N=1,2,\ldots,
\end{equation}
subject to the random\index{random initial conditions} initial conditions represented by a strictly
$\varphi $-sub-Gaus\-sian harmonizable processes were considered in
\cite{BKOS,KozOrsSakhVas_JOSS2018}. Rigorous conditions were stated
therein for the existence of solutions and some distributional
properties of solutions were investigated.

The present paper continues the line of research initiated in the papers
\cite{BKOS,KozOrsSakhVas_JOSS2018}. Note that in the mathematical
literature the initial value problems for partial differential equations
have been studied within the framework of different functional spaces,
including the most abstract ones. Here we take into consideration Equation \eqref{1.8} in the framework of special Banach spaces of random
variables,\index{random variables} which constitute a subclass of Orlicz spaces of exponential
type, more precisely, we deal with the spaces of strictly
$\varphi $-sub-Gaussian random processes. These spaces play an important
role in extensions of properties of Gaussian and sub-Gaussian processes.
Basic results on $\varphi $-sub-Gaussian processes and fields can be
found, for example, in \cite{BK,GKN,KO,VKY2005}.

The general methods and techniques developed for $\varphi $-sub-Gaussian
processes, applied to the problems under consideration in the present
paper, permit us to obtain bounds for the distributions of suprema of
the solutions to the initial value problem for Equation
\eqref{1.8}. The bounds are presented in a form different than those
obtained in the paper \cite{KozOrsSakhVas_JOSS2018}, and can be
more useful in particular situations. In such a way, the results of the
present paper complement and extend the results presented in
\cite{KozOrsSakhVas_JOSS2018}.

To make the paper self-contained, in Sections \ref{prelim} we present all important
definitions and facts on harmonizable $\varphi $-sub-Gaussian processes,
which will be used in the derivation of the main results. We also
formulate the result on the conditions of existence of solutions to
\eqref{1.8} with the $\varphi $-sub-Gaussian initial condition (see
\cite{BKOS,KozOrsSakhVas_JOSS2018} for its derivation). The main result
on the bounds for the distributions of supremum of the solutions is
stated in Section \ref{main}. We present examples of processes for which the
assumptions of the general result are verified and bounds are written
explicitly. The main result is also specified for the case of Gaussian
initial condition.

\section{Preliminaries}%
\label{prelim}

Since in the paper we consider a partial differential equation with
random\index{random initial conditions} initial condition given by a real-valued harmonizable
$\varphi $-sub-Gaussian process, in this section we present the
necessary definitions and facts concerning such processes.

\subsection{Harmonizable processes\index{harmonizable processes}}%
\label{harmonizable}

Harmonizable processes\index{harmonizable processes} are a natural extension of stationary processes\index{stationary processes}
to second-order nonstationary ones. Such class of processes allows us
to retain advantages of the Fourier analysis. Harmonizable processes\index{harmonizable processes}
were introduced by Lo\`{e}ve \cite{L}. Recent developments on this
theory are due to Rao \cite{Rao} and Swift \cite{Swift}
among others.

\begin{defi}[\cite{L}] The second-order random function $X=\{X(t),t\in
\mathbb{R}\}$,\break  $ \mathsf{E}X(t)=0$, is called harmonizable if there
exists a second-order random function $y={y(t), t\in \mathbb{R}}$,
$\mathsf{E}y(t)=0$ such that the covariance $\Gamma _{y} (t,s)=
\mathsf{E}y(t)\overline{y(s)}$ has finite variation\index{finite variation} and $X(t)=
\int _{\mathbb{R}}e^{itu}\,\mathrm{d}y(u)$, where the integral is defined
in the mean-square sense.
\end{defi}

\begin{thm}[\cite{L} Lo\`{e}ve theorem]
\label{Loev_thm}
The second-order random function
$X=\{X(t),\allowbreak t\in \mathbb{R}\}$, $ \mathsf{E}X(t)=0$, is harmonizable if
and only if there exists a covariance function $\Gamma _{y} (u,v)$ with
finite variation\index{finite variation} such that
%
\begin{equation}
\label{Gamma_x}
\Gamma _{x}(t,s)= \mathsf{E}X(t) \overline{X(s)}=
\int _{\mathbb{R}}
\int _{\mathbb{R}}e^{i(tu-sv)}\,\mathrm{d}\Gamma _{y} (u,v).
\end{equation}
\end{thm}

\begin{rem}
In the theorem above, the covariance function $\Gamma _{y}$ is of bounded
Vitali variation (see \cite{Rao,Swift}). This fact guarantees that
integral in (\ref{Gamma_x}) is in the Lebesgue sense. The function
$\Gamma _{y}$ is also called the spectral function or bi-measure of the
process~$X$.
\end{rem}

\begin{rem}
In what follows, an integral of the type $\int \int  _{A}f(t,s)
\mathrm{d}g(t,s)$ is understood as a common Le\-bes\-gue--Stiel\-tjes
integral, that is, the limit of the sum $\sum \sum f(t,\break s)\Delta _{i}
\Delta _{j} g(t,s)$, and an integral of the type $\int \int  _{A}f(t,s)|
\mathrm{d}g(t,s)|$ is is understood as the limit of the sum
$\sum \sum f(t,s)|\Delta _{i}\Delta _{j} g(t,s)|$.
\end{rem}

Below we shall focus on real-valued harmonizable processes.

\begin{defi}
\label{real_harmonizable}
Real-valued second order random function $X=\{X(t),t\in \mathbb{R}\}$
is called harmonizable, if there exists a real-valued second order
function $y(u)$, $ \mathsf{E}y(u)\,{=}\break0$, $u\in \mathbb{R}$, such that
$X(t)=\int _{-\infty }^{\infty } \sin tu \,\mathrm{d}y(u)$ or
$X(t)=\int _{-\infty }^{\infty } \cos tu \,\mathrm{d}y(u)$ and the
covariance function $\Gamma _{y}(t,s)=\mathsf{E}y(t)y(s)$ has finite
variation.\index{finite variation} The integral is defined in the mean-square sense.
\end{defi}

\begin{thm}
\label{Loev_thm_real}
Real-valued second order function $X = \{X(t), t\in \mathbb{R}\}$,
$ \mathsf{E}X(t)=0 $, is harmonizable if and only if there exists the
covariance function $\Gamma _{y}(u,v)$ with finite variation\index{finite variation} such that
\begin{equation*}
\Gamma _{x}(t,s)=\mathsf{E}X(t)X(s)=\int _{\mathbb{R}} \int _{\mathbb{R}}
\kappa (t u) \kappa (s v) \,\mathrm{d}\Gamma _{y} (u,v),
\end{equation*}
where $\kappa (v) = \cos v$ or $\kappa (v) = \sin v$.
\end{thm}
The theorem above follows from Theorem~\ref{Loev_thm}.

\subsection{$\varphi $-sub-Gaussian random variables and processes}%
\label{phisubgaussian}

Here we present some basic facts from the theory of
$\varphi $-sub-Gaussian random variables and processes, as well as some
necessary results.

A continuous even convex function $\varphi $ is called an \textit{Orlicz
$N$-function}, if $\varphi (0)=0, \varphi (x)>0$ as $x\neq 0$, and the
following conditions hold:\vadjust{\goodbreak} $\lim _{\substack{ x\rightarrow 0\\}}\frac{\varphi (x)}{x}=0$,\break
$\lim _{\substack{ x\rightarrow \infty\\}}\frac{\varphi (x)}{x}=\infty $. The function $\varphi ^{*}$ defined by
$\varphi ^{*}(x)=\sup _{\substack{ y\in \mathbb{R}\\}}(xy-\varphi (y))$ is called \textit{the Young--Fenchel transform (or
convex conjugate)} of the function $\varphi $ \cite{BK,Kr.Rut}.

We say that $\varphi $ satisfies the \textit{Condition $Q$}, if
$\varphi $ is such an $N$-function that\break  $\liminf
 _{\substack{ x\rightarrow 0\\}}\frac{\varphi (x)}{x^{2}}=c>0$, where the case $c=\infty $ is possible
\cite{GKN,KO}.

In what follows we will always deal with $N$-functions for which
condition $Q$ holds.

Examples of N-functions, for which the condition $Q$ is satisfied:
%
\begin{align}
&\varphi (x)=\frac{|x|^{\alpha }}{\alpha }, \quad 1 <\alpha \leq 2,
\nonumber
\\
&\varphi (x)=
{
\begin{cases}
\frac{|x|^{\alpha }}{\alpha }, \quad |x|\geq 1,\alpha > 2;
\\
\frac{|x|^{2}}{\alpha }, \quad |x|\leq 1,\alpha > 2.
\\
\end{cases}
}
\label{ex_phi1}
\end{align}

\begin{defi}[\cite{GKN,KO}] Let $\{\Omega , L,\mathbf{P}\}$ be a standard
probability space. The random variable\index{random variables} $\zeta $ is said to be
$\varphi $-sub-Gaussian (or belongs to the space $\mathrm{Sub}_{
\varphi }(\Omega )$), if $\mathsf{E}\zeta =0$, $ \mathsf{E}\exp \{
\lambda \zeta \}$ exists for all $\lambda \in \mathbb{R}$ and there
exists a constant $a>0$ such that the following inequality holds for all
$\lambda \in \mathbb{R}:$ $\mathsf{E}\exp \{\lambda \zeta \} \leq
\exp \{\varphi (\lambda a)\}$.
\end{defi}

The space $ \mathrm{Sub}_{\varphi }(\Omega )$ is a Banach space with
respect to the norm \cite{GKN,KO}
\begin{equation*}
\tau _{\varphi }(\zeta ) = \inf \{a > 0: \mathsf{E}\exp \{\lambda
\zeta \} \leq \exp \{ \varphi (a \lambda )\}\},
\end{equation*}
which is called the $\varphi $-sub-Gaussian standard of the random
variable\index{random variables} $\zeta $.

Centered Gaussian random variables $\zeta \sim N(0,\sigma ^{2})$\index{random variables} are
$\varphi $-sub-Gaussian with $\varphi (x)=\frac{x^{2}}{2}$ and
$\tau ^{2}_{\varphi }(\zeta )=\mathsf{E}\zeta ^{2}=\sigma ^{2}$. In the
case $\varphi (x)=\frac{x^{2}}{2}$, $\varphi $-sub-Gaussian random
variables are simply sub-Gaussian.

\begin{defi}
A family $\Delta $ of random variables\index{random variables} ${\zeta } \in \mathrm{Sub}_{
\varphi }(\Omega )$ is called strictly $\varphi $-sub-Gaussian (see
\cite{KK}), if there exists a constant $C_{\Delta }$ such that for all
countable sets $I$ of random variables\index{random variables} ${\zeta _{i}}\in \Delta $,
$i\in I$, the following inequality holds:
%
\begin{equation}
\label{*}
\tau _{\varphi }\left ( \sum _{i \in I} \lambda _{i} \zeta _{i} \right )
\leq C_{\Delta }\left ( \mathsf{E}\left ( \sum _{i \in I}\lambda _{i}
\zeta _{i} \right )^{2}\right )^{1/2}.
\end{equation}
\end{defi}

The constant $C_{\Delta }$ is called the \textit{determining} constant
of the family $\Delta $.

The linear closure of a strictly $\varphi $-sub-Gaussian family
$\Delta $ in the space $L_{2}(\Omega )$ is the strictly
$\varphi $-sub-Gaussian with the same determining constant
\cite{KK}.

\begin{defi}
The random process\index{random process} $\zeta =\{\zeta (t), t\in \mathbf{T}\}$ is called
(strictly) $\varphi $-sub-Gaussian if the family of random variables\index{random variables}
$\{\zeta (t), t\in \mathbf{T}\}$ is (strictly) $\varphi $-sub-Gaussian
\cite{KK}.
\end{defi}

The following example of strictly $\varphi $-sub-Gaussian random process
is important for our study. The solutions of partial differential
equations considered in the next sections are of the same form as in
this example.

\begin{ex}[\cite{KK}]
\label{ex_kernel}%
Let $K$ be a deterministic kernel and suppose that the
process $X= \{  X(t), t\in \mathbf{T} \}  $ can be
represented in the form
\begin{equation*}
X(t) =\int \limits _{T} K(t,s) \,\mathrm{d}\xi (s),
\end{equation*}
where $\xi (t)$, $t\in \mathbf{T}$, is a strictly
$\varphi $-sub-Gaussian random process and the integral above is defined
in the mean-square sense. Then the process $X(t)$, $t\in \mathbf{T}$,
is a strictly $\varphi $-sub-Gaussian random process with the same
determining constant.
\end{ex}

The notion of admissible function\index{admissible function} for the space $\mathrm{Sub}_{\varphi
}(\Omega )$ will be used to state the conditions of the existence of
solutions of partial differential equations considered in the paper and
to write down the bounds for suprema\index{bounds for suprema} of these solutions.

\begin{lmm}[\cite{KR,KRP}, p. 146]
\label{lmm12}%
Let $Z(u),u\geq 0$, be a
continuous, increasing function such that $ Z(u)>0$ and the function
$\frac{u}{Z(u)}$ is non-decreasing for $u>u_{0}$, where $u_{0}\geq 0$
is a constant. Then for all $u,v\neq 0$
%
\begin{equation}
\label{**}
\left |\sin \frac{u}{v}\right |\leq \frac{Z\left (\left |u\right |+u
_{0}\right )}{Z\left (\left |v\right |+u_{0}\right )}
\end{equation}
\end{lmm}

\begin{defi}
Function $Z(u), u\geq 0$, is called \textit{admissible} for the space
$\mathrm{Sub}_{\varphi }(\Omega )$, if for $Z(u)$ conditions of
Lemma~\ref{lmm12} hold and for some $\varepsilon >0$ the integral
\begin{equation*}
\int _{0}^{\varepsilon }\Psi \left (\ln \left (Z^{(-1)}\left (\frac{1}{s}
\right )-u_{0}\right )\right )ds
\end{equation*}
converges, where $\Psi (v)=\frac{v}{\varphi ^{(-1)}(v)}, v>0$.
\end{defi}

For example, the function $ Z(u) = u^{\rho }$, $ 0 < \rho <1$, is
an admissible\index{admissible function} function for the space $\mathrm{Sub}_{\varphi }(\Omega
)$ if $\varphi $ is defined in (\ref{ex_phi1}).

Characteristic feature of $\varphi $-sub-Gaussian random variables is
the exponential\break  bounds\index{exponential bounds} for their tail probabilities. For
$\varphi $-sub-Gaussian processes estimates for their suprema are
available in different forms, see, for example, the book~\cite{BK}.

To derive our main results, we shall use the following theorem on the
distribution of supremum of a $\varphi $-sub-Gaussian random process,
proved in the paper~\cite{KO2016} (see also \cite{BK}).

\begin{thm}
\label{Theorem5_KO2016}
Let $X = \{ X(t), \, t \in \mathbf{T}\}$ be a $\varphi $-sub-Gaussian
process and $\rho _{X}$ be the pseudometrics generated by $X$, that is,
$ \rho _{X} (t, s) = \tau _{\varphi }(X(t) - X(s))$, $ t, s \in
\mathbf{T} $. Assume that the pseudometric space $(\mathbf{T}, \rho
_{X})$\index{pseudometric space} is separable, the process $X$ is separable on $(\mathbf{T},
\rho _{X})$ and $\varepsilon _{0} :=\sup  _{t\in \mathbf{T}}
\tau _{\varphi }(X(t)) < \infty $.

Let $r(x),x\ge 1$, be a non-negative, monotone increasing function such
that the function $r(e^{x}),x\ge 0$, is convex, and for $0<\varepsilon
\le \varepsilon _{0}$
%
\begin{equation}
\label{cond_I_r}
I_{r}(\varepsilon ):=\int \limits _{0}^{\varepsilon }
r (N(v))\,d v <
\infty ,
\end{equation}
where $\{N(v),v>0\}$ is the massiveness of the pseudometric space
$(\mathbf{T}, \rho _{X})$,\index{pseudometric space} that is, $N(v)$ denotes the smallest number
of elements in a $v$-covering of $T$, and the covering is formed by
closed balls of radius of at most $v$.

Then for all $\lambda >0$, $0<\theta <1$ and $u>0$ it holds
%
\begin{equation}
\mathsf{E}\exp \left \{
\lambda \sup \limits _{t\in \mathbf{T}}|X(t)|\right \}  \le Q(\lambda ,
\theta )
\end{equation}
and
%
\begin{equation}
P \{\sup \limits _{t\in \mathbf{T}} |X(t)| \geq u\} \leq 2A(\theta ,u),
\end{equation}
where
\begin{align*}
Q(\lambda ,\theta )&:=\exp \left \{  \varphi \left (\frac{\lambda
\varepsilon _{0}}{1-\theta }\right )\right \}  r^{(-1)}\left (\frac{I
_{r}(\theta \varepsilon _{0})}{\theta \varepsilon _{0}}\right ),
\\
A(\theta ,u) &= \exp \Big \{- \varphi ^{*} \Big ( \frac{u (1-\theta )}{
\varepsilon _{0}} \Big ) \Big \}r^{(-1)}\left (\frac{I_{r}(\theta
\varepsilon _{0})}{\theta \varepsilon _{0}}\right ).
\end{align*}
\end{thm}

\begin{rem}
The integrals of the form
%
\begin{equation}
\label{entropy}
I(\varepsilon ):=\int \limits _{0}^{\varepsilon }
g (N(v))\,d v , \quad
\varepsilon >0,
\end{equation}
with $g(v)$, $v\ge 1$, being a nonnegative nondecreasing function, are
called entropy integrals. Entropy characteristics of the parametric set
$T$ with respect to the pseudometrics $ \rho _{X} (t, s) = \tau _{
\varphi }(X(t) - X(s))$, $ t, s \in \mathbf{T}$, generated by the
process $X=\{X(t), t\in \mathbf{T}\}$, and the rate of growth of metric
massiveness\index{metric massiveness} $N(v)$, or metric entropy\index{metric entropy} $H(v):=\ln (N(v))$, are closely
related to the properties of the process $X$ (see \cite{BK} for
details).\looseness=-1

The integrals \eqref{entropy} play an important role in the study of
such properties as boundedness and continuity of sample paths of a
process, these integrals appear in estimates for modulii of continuity
and distribution of supremum.

General results of this kind for $\varphi $-sub-Gaussian processes are
related to the convergence of the integrals \eqref{entropy}, where for
$g(v)$ one takes $\Psi (\ln (v))$ with $\Psi (v)=
\frac{v}{\varphi ^{(-1)}(v)}, v>0$. Theorem \ref{Theorem5_KO2016} is more
suitable for the case of ``moderate'' growth of the metric entropy\index{metric entropy} and
can lead to improved inequalities for upper bound for the distribution
of supremum of the process, in comparison with more general inequalities
involving the integrals based on the above function $\Psi $ (see
\cite{BK}).

Entropy methods are also used in the modern approximation theory.
Theorem \ref{Theorem5_KO2016} was applied, for example, in
\cite{KO2016}, for developing uniform approximation schemes for
$\varphi $-sub-Gaussian processes.

\end{rem}

\subsection{Solutions of linear odd-order heat-type equations with
random\index{random initial conditions} initial conditions}%
\label{subsection_solution}

Let us consider the linear equation
%
\begin{equation}
\label{linear_eqn}
\sum _{k=1}^{N} a_{k}\frac{\partial ^{2k+1} U(t,x)}{\partial x^{2k+1}}=
\frac{\partial U(t,x)}{\partial t} ,\quad  t>0,\, x \in \mathbb{R},
\end{equation}
subject to the random\index{random initial conditions} initial condition
%
\begin{equation}
\label{random_ini_cond}
U(0,x)=\eta (x),\quad  x \in \mathbb{R},
\end{equation}
and $\{a_{k}\}_{k=1}^{N}$ are some constants.

The next theorem gives the conditions of the existence of the solutions
of the equation above with a $\varphi $-sub-Gaussian initial condition
$\eta (x)$ (see \cite{BKOS,KozOrsSakhVas_JOSS2018}).

\begin{thm}
\label{thm_solution}
Let $\eta =\{\eta (x),x\in \mathbb{R}\}$ be a real harmonizable (see
Definition~\ref{real_harmonizable}) and strictly $\varphi $-sub-Gaussian
random process.\vadjust{\goodbreak} Also let $Z=\{Z(u),u\geq 0\}$ be a function admissible\index{admissible function}
for the space $\mathrm{Sub}_{\varphi }(\Omega )$. Assume that the
following integral converges
%
\begin{equation}
\label{exists_classic}
\int _{\mathbb{R}}\int _{\mathbb{R}}\left |\lambda \right |^{2N+1}
\left |\mu \right |^{2N+1}Z\left (u_{0}+\left |\lambda \right |^{2N+1}
\right )Z\left (u_{0}+\left |\mu \right |^{2N+1}\right )d|\Gamma _{y}(
\lambda ,\mu )|<\infty .
\end{equation}
Then
%
\begin{equation}
\label{solution}
U(t,x)=\int _{-\infty }^{\infty }I(t,x,\lambda )\, dy(\lambda )
\end{equation}
is the classical solution to the problem
(\ref{linear_eqn})--(\ref{random_ini_cond}), that is, $U(t,x)$ satisfies
\xch{Equation}{equation} (\ref{linear_eqn}) with probability one and $U(0,x)=\eta (x)$.
Here
%
\begin{equation}
\label{Itxlm}
I(t,x,\lambda )=\kappa \left (\lambda x +t\sum _{k=1}^{N} a_{k}
\lambda ^{2k+1}(-1)^{k} \right ),
\end{equation}
and $\kappa (v) = \cos v$ or $\kappa (v) = \sin v$ for the cases when
$\eta (x) = \int  _{\mathbb{R}}cos (xu)\,dy(u)$ or $\eta (x) =
\int  _{\mathbb{R}}\sin (xu) \,dy(u)$ respectively.
\end{thm}

\begin{rem}
Note that under the condition \eqref{exists_classic} all the integrals
%
\begin{equation}
\label{integr}
\int _{\mathbb{R}}\lambda ^{s}I\left ( t,x,\lambda \right ) dy\left (
\lambda \right ) ,\quad  s=0,1,2,\ldots ,2N+1,
\end{equation}
converge uniformly in probability for $|x| \leq A$ and $0 \leq t
\leq T$ for all $A>0$, $ T>0$. This guarantees that the derivatives of
orders $s=1,2,\ldots ,2N+1$ of solution $U(t,x)$ given by
\eqref{solution} exist with probability one. In this sense we can treat
$U(t,x)$ as a classical solution. We refer for more details to
\cite{BKOS}.
\end{rem}

\begin{rem}
Similar result can be stated for the case $\eta (x)=\int _{-\infty }
^{\infty }(a \sin xu + b\cos xu) \,\mathrm{d}y(u)$, where $a$ and
$b$ are some real constants.
\end{rem}

\begin{rem}
Let $\varphi  ( x ) =\frac{ | x | ^{p}}{p}$,
$p>1$, for sufficiently large $x$. Then the statement of Theorem
\ref{thm_solution} holds if the following integral converges
%
\begin{equation}
\label{phixp}
\int _{\mathbb{R}}\int _{\mathbb{R}}\left | \lambda \mu \right | ^{2N+1}
\left ( \ln \left (
1+\lambda \right ) \ln \left ( 1+\mu \right )
\right ) ^{\alpha }d|\Gamma _{y}\left ( \lambda ,\mu \right )| ,
\end{equation}
where $\alpha $ is a constant such that $\alpha >1-\frac{1}{p}$ (see
\cite{BKOS}).
\end{rem}

\section{Main results}%
\label{main}

\subsection{Some auxiliary estimates}%
\label{auxiliary}

Let us consider a separable metric space $(\mathbf{T},d)$, where
$\mathbf{T}= \{ a_{i} \leq t_{i} \leq b_{i},\, i=1, 2 \}$ and
$ d({t}, {s}) = \max  _{i=1,2} |t_{i} -s_{i}|$, ${t} =(t_{1}, t
_{2})$, ${s} =(s_{1}, s_{2})$.

\begin{thm}
\label{sup_X}
Let $ X=\{X({t}), {t} \in \mathbf{T}\}$ be a separable
$\varphi $-sub-Gaussian random process such that $\varepsilon _{0} =
\sup  _{t\in \mathbf{T}} \tau _{\varphi }(X(t)) < \infty $ and
%
\begin{equation}
\label{suptau3.16}
\sup _{\substack{
d({t},{s})\leq h,
\\
{t}, {s}\in \mathbf{T}
\\
}}\tau _{\varphi } (X ({t})-
X ({s}))\leq \sigma (h),
\end{equation}
where $\{\sigma (h),\; 0 < h \leq \max  _{i = 1,2} |b_{i} -a_{i}|
\}$ is a monotonically increasing continuous function such that
$\sigma (h)\rightarrow 0$ as $ h \rightarrow 0 $, and for $0<\varepsilon
\le \gamma _{0}$
%
\begin{equation}
\label{intPsi3.16}
\tilde{I}_{r}(\varepsilon ):=\int \limits _{0}^{\varepsilon }
r \left (\left (\frac{b
_{1}-a_{1}}{2\sigma ^{(-1)}(v)}+1\right )\left (\frac{b_{2}-a_{2}}{2
\sigma ^{(-1)}(v)}+1\right )\right )\,d v < \infty ,
\end{equation}
where $ \gamma _{0} = \sigma (\max _{i =1,2} |b_{i} -a_{i}| ) $ and
$r(x),x\ge 1$, is defined in Theorem~\ref{Theorem5_KO2016}. Then for all
$0<\theta <1$ and $u>0$ it holds
%
\begin{equation}
\label{eq3.3}
P \{\sup \limits _{t\in \mathbf{T}} |X(t)| \geq u\} \leq 2\tilde{A}(
\theta ,u),
\end{equation}
where
\begin{equation*}
\tilde{A}(\theta ,u) = \exp \Big \{- \varphi ^{*} \Big ( \frac{u (1-
\theta )}{\varepsilon _{0}} \Big ) \Big \}r^{(-1)}\left (\frac{
\tilde{I}_{r}(\min (\theta \varepsilon _{0}, \gamma _{0}))}{\theta
\varepsilon _{0}}\right ).
\end{equation*}
\end{thm}
\begin{proof}
To prove this theorem, we apply Theorem~\ref{Theorem5_KO2016} in the
case of the separable metric space $(\mathbf{T},d)$ with $T = \{ a
_{i} \leq t_{i} \leq b_{i},\, i=1, 2 \}$ and $ d({t}, {s}) = \max
 _{i=1,2} |t_{i} -s_{i}|$,\break  ${t} =(t_{1}, t_{2})$, ${s} =(s_{1},
s_{2})$.

In particular, condition~(\ref{suptau3.16}) means that
\begin{equation*}
\sup _{\substack{
d({t},{s})\leq h,
\\
{t}, {s}\in \mathbf{T}
\\
}}\tau _{\varphi } (X ({t})-
X ({s}))=
\sup _{\substack{
d({t},{s})\leq h,
\\
{t}, {s}\in \mathbf{T}
\\
}}\rho _{X}(t,s)\leq \sigma (h),\quad  0 < h \leq \max \limits _{i = 1,2} |b
_{i} -a_{i}|.
\end{equation*}
From the fact that the process $X$ is separable on $(\mathbf{T}, \rho
_{X})$ and the function $\{\sigma (h),\; 0 < h \leq \max  _{i =
1,2} |b_{i} -a_{i}|\}$ is a monotonically increasing continuous
function, we get that for $\varepsilon \le \gamma _{0}$ the smallest
number of elements in an $\varepsilon $-covering of the pseudometric
space $(\mathbf{T}, \rho _{X})$\index{pseudometric space} can be estimated as the smallest number
of elements in a $(\sigma ^{(-1)}(\varepsilon ))$-covering of the metric
space $(\mathbf{T}, d)$ as follows:
\begin{equation*}
N(\varepsilon ) \leq \Big [ \Big ( \frac{b_{1} -a_{1}}{2 \sigma ^{(-1)} (
\varepsilon )}
+ 1\Big ) \Big ( \frac{b_{2} -a_{2}}{2 \sigma ^{(-1)} (
\varepsilon )}+ 1\Big ) \Big ]
\quad \text{and} \quad N(\varepsilon ) =
1 \quad \text{if}\ \varepsilon > \gamma _{0}.
\end{equation*}

Hence, from condition~(\ref{intPsi3.16}) we get that
\begin{equation*}
I_{r}(\varepsilon )=\int \limits _{0}^{\varepsilon }
r (N(v))\,d v
\le \tilde{I}_{r}(\varepsilon )<\infty ,\quad \varepsilon \le \gamma
_{0},
\end{equation*}
that is, conditions of Theorem~\ref{Theorem5_KO2016} are satified for
the process $X$. Finally, taking into account the estimates above and
the properties of the function $r$ we derive the estimate for
distribution of supremum of the process $X$ for all $0<\theta <1$ and
$u>0$:
\begin{gather*}
P \{\sup \limits _{t\in \mathbf{T}} |X(t)| \geq u\} \leq 2\exp \Big \{-
\varphi ^{*} \Big ( \frac{u (1-\theta )}{\varepsilon _{0}} \Big ) \Big \}r
^{(-1)}\left (\frac{I_{r}(\min (\theta \varepsilon _{0}, \gamma _{0}))}{
\theta \varepsilon _{0}}\right )
\\
\leq 2\exp \Big \{- \varphi ^{*} \Big (
\frac{u (1-\theta )}{\varepsilon _{0}} \Big ) \Big \}r^{(-1)}\left (\frac{
\tilde{I}_{r}(\min (\theta \varepsilon _{0}, \gamma _{0}))}{\theta
\varepsilon _{0}}\right ) = 2\tilde{A}(\theta ,u).\qedhere
\end{gather*}
\end{proof}

As can be seen from Theorem \ref{sup_X}, it is crucial to guarantee some
kind of continuity of $X$ on $T$ in the form \eqref{suptau3.16}, that
is, with respect to the norm $\tau _{\varphi }$ induced by the process
$X$ itself. Fulfilment of condition \eqref{suptau3.16} enables us to
write down the upper bound \eqref{eq3.3} for the distribution of
supremum of a $\varphi $-sub-Gaussian process $ X=\{X({t}), {t}
\in \mathbf{T}\}$ defined on a separable metric space~$(\mathbf{T}, d)$.

Below we present as a separate theorem the very useful result giving the
conditions for the estimate \eqref{suptau3.16} to hold for the field
$\{U(t,x),\,a \leq t \leq b,\, c \leq x \leq d\}$ representing the
solution to (\ref{linear_eqn})--(\ref{random_ini_cond}).

\begin{thm}
\label{thm_sigma}
Let $y=\{y(u),\, u \in \mathbb{R}\}$ be a strictly
$\varphi $-sub-Gaussian random process with a determining constant
$C_{y}$ and $ U(t, x) = \int  _{-\infty }^{\infty }I(t, x,
\lambda )\,dy(\lambda )$, where $I(t, x, \lambda )$ is given in
Theorem~\ref{thm_solution}, $a \leq t \leq b$, $c \leq x \leq d$. Assume
that $U(t, x)$ exists and is continuous with probability one (this
condition holds if Theorem~\ref{thm_solution} holds). Let
$\mathsf{E}y(t) y(s) = \Gamma _{y} (s,t)$. Assume that $\{Z(u),u\ge 0
\}$ is an admissible\index{admissible function} function for the space $\mathrm{Sub}_{\varphi }(
\Omega )$. If the integral
%
\begin{align}
\label{CZ}
C_{Z}^{2} &= \int _{-\infty }^{\infty }\int _{-\infty }^{\infty }\left (Z
\Big (\frac{|\lambda |}{2}+u_{0}\Big ) + Z \Big ( \frac{1}{2}
\Big |
\sum _{k=1}^{N} a_{k} \lambda ^{2k+1} (-1)^{k}\Big | + u_{0} \Big )\right )
\nonumber
\\&
\times \left (Z \Big (\frac{|\mu |}{2}+u_{0}\Big ) + Z \Big (
\frac{1}{2}
\,\Big | \sum _{k=1}^{N} a_{k} \mu ^{2k+1} (-1)^{k}\Big | + u
_{0} \Big )\right )\,d|\Gamma _{y}(\lambda , \mu )|
\end{align}
converges, then there exists the function
%
\begin{equation}
\label{sigma}
\sigma (h ) = 2 C_{y} C_{Z} \Big ( Z\Big ( \frac{1}{h} +u_{0} \Big )
\Big )^{-1},\quad 0<h<\max (b-a,d-c),
\end{equation}
such that
%
\begin{equation}
\label{cond_sigma}
\sup _{ \substack{t,t_{1}\in [a,b]:|t- t_{1}| \leq h
\\
x,x_{1}\in [c,d]:|x-x_{1}| \leq h
\\
}}
\tau _{\varphi }( U(t, x) - U(t_{1}, x_{1})) \le \sigma (h).
\end{equation}
\end{thm}

This result was obtained in the paper~\cite{KozOrsSakhVas_JOSS2018}
as an intermediate statement in the course of the proof of Theorem 5.1.
To make the present paper self-contained, we present in Appendix the
main steps of the proof.

\subsection{On the distribution of supremum of solution to the
problem (\ref{linear_eqn})--(\ref{random_ini_cond})}

Now we have all the necessary tools to derive the estimate for the
distribution of supremum of the field $U(t,x)$ representing the solution
to (\ref{linear_eqn})--(\ref{random_ini_cond}).

\begin{thm}
\label{thm_sup_solution}
Let $y=\{y(u),\, u \in \mathbb{R}\}$ be a strictly
$\varphi $-sub-Gaussian random process with a determining constant
$C_{y}$ and $ U(t, x) = \int  _{-\infty }^{\infty }I(t, x,
\lambda )\,dy(\lambda )$, where $I(t, x, \lambda )$ is given in
Theorem~\ref{thm_solution}, $a \leq t \leq b$, $c \leq x \leq d$. Assume
that for $U(t, x)$ the conditions of Theorem~\ref{thm_sigma} hold. Let
$r(x)$, $x\ge 1$, be a non-negative, monotone increasing function such
that the function $r(e^{x}),x\ge 0$, is convex, and condition
\eqref{intPsi3.16} is satisfied for $\sigma (h)$ given by
\eqref{sigma}.

Then for all $0<\theta <1$ and $ u > 0$ the following inequality holds
true
%
\begin{equation}
\label{P_sup_U}
P \big \{ \sup _{\substack{a \le t \le b
\\
c \le x \le d}}
|U(t, x)| > u \big \} \leq 2\hat{A}(\theta ,u),
\end{equation}
where
%
\begin{align}
\hat{A}(\theta ,u) &= \exp \Big \{- \varphi ^{*} \Big ( \frac{u (1-
\theta )}{\varepsilon _{0}} \Big ) \Big \}r^{(-1)}\left (\frac{\hat{I}
_{r}(\min (\theta \varepsilon _{0}, \gamma _{0}))}{\theta \varepsilon
_{0}}\right ),\\
\hat{I}_{r} (\delta )
& = \int _{0}^{\delta }r \left ( \left [ \Big (
\frac{b-a}{2} \Big ( Z^{(-1)}
\Big ( \frac{2 C_{Z} C_{y}}{s} \Big ) - u
_{0} \Big ) +1 \Big ) \right .\right .
\nonumber
\\
& \left .\left .\times \Big ( \frac{d-c}{2} \Big ( Z^{(-1)} \Big ( \frac{2
C_{Z}
C_{y}}{s} \Big ) -u_{0} \Big ) +1\Big )\right ] \right ) \,ds,\\
\varepsilon _{0} &=\sup _{\substack{a \le t \le b
\\
c \le x \le d}}\tau _{\varphi }(U(t, x)),\nonumber
\\
\gamma _{0} &= \frac{2 C_{y} C_{Z}}{Z( \frac{1}{\varkappa } +u_{0})},
\quad \varkappa = \max (b-a, d-c).\nonumber
\end{align}%
\end{thm}

\begin{proof}
The assertion of this theorem follows from Theorems~\ref{sup_X} and
\ref{thm_sigma}. Since the conditions of Theorem~\ref{thm_sigma} are
satisfied, then there exists the function $ \sigma (h ) = 2 C_{y} C
_{Z}  \bigl( Z \bigl( \frac{1}{h} +u_{0}  \bigr)  \bigr)^{-1}, 0<h<
\max (b-a,d-c) $, such that
\begin{equation*}
\sup _{ \substack{t,t_{1}\in [a,b]:|t- t_{1}| \leq h
\\
x,x_{1}\in [c,d]:|x-x_{1}| \leq h
\\
}}
\tau _{\varphi }( U(t, x) - U(t_{1}, x_{1})) \le \sigma (h).
\end{equation*}
In this case,
\begin{equation*}
\sigma ^{(-1)} (v) = \Big ( Z^{(-1)} \Big ( \frac{2 C_{y} C_{Z}}{v}
\Big ) - u_{0} \Big )^{-1},
\quad 0 < v < \frac{2 C_{y} C_{Z}}{Z (\frac{1}{
\varkappa } + u_{0})} =
\gamma _{0},
\end{equation*}
and for $\varepsilon _{0}$ the upper bound \eqref{A.1} holds (see Appendix A).

Since the conditions \eqref{suptau3.16} and \eqref{intPsi3.16} of
Theorem \ref{sup_X} also hold true, the final estimate directly follows.
\end{proof}

\begin{rem}
The derivation of our main result is based on Theorem
\ref{Theorem5_KO2016}, and due to this we present the bounds for the distribution of the
supremum of the process $U(t,x)$ in the form different than those
obtained in the paper \cite{KozOrsSakhVas_JOSS2018}. This form of
bounds can be more useful in particular situations giving the
possibility to calculate the explicit expressions for the bounds.
\end{rem}

Now we will specify the statement of Theorem \ref{thm_sup_solution} for
particular choices of the admissible function $Z$ and the function
$\varphi $.

\begin{ex}
\label{ex1}
Consider $\varphi (x)=\frac{|x|^{\alpha }}{\alpha }$, $1 <\alpha
\leq 2$. Then $\varphi ^{*}(x)=\frac{|x|^{\gamma }}{\gamma }$, where
$\gamma \ge 2$, and $\frac{1}{\alpha }+\frac{1}{\gamma }=1$. Hence,
\begin{equation*}
\hat{A}(\theta ,u) = \exp \Big \{- \frac{u^{\gamma }(1-\theta )^{
\gamma }}{\gamma \varepsilon _{0}^{\gamma }} \Big \}r^{(-1)}\left (\frac{
\hat{I}_{r}(\min (\theta \varepsilon _{0}, \gamma _{0}))}{\theta
\varepsilon _{0}}\right ).
\end{equation*}

Now it is necessary to make the following steps:
\begin{enumerate}%
\item
to check the fulfilment of \eqref{CZ} with a particular function $Z$,
admissible for a given $\varphi $;
\item
to calculate $\sigma (h)$ in \eqref{sigma};\vadjust{\goodbreak}
\item
to choose the function $r$ satisfying \eqref{intPsi3.16};
\item
to derive an estimate for $\hat{A}(\theta ,u)$.
\end{enumerate}

So, let us choose the admissible function $Z(u)=|u|^{\rho }$,
$0 <\rho \le 1$.

In this case $ u_{0} = 0$, $Z^{(-1)} (u) = u^{\frac{1}{\rho }}$,
$u > 0$, and
%
\begin{align}
C_{Z}^{2}
& = \int _{-\infty }^{\infty }\int _{-\infty }^{\infty }
\Big (
\Big (\frac{|\lambda |}{2} \Big )^{\rho }+ \Big | \frac{1}{2}
\sum _{k=1}^{N} a_{k}
\lambda ^{2k+1} (-1)^{k} \Big |^{\rho }\Big )
\label{CZsqrt}
\\
& \times \Big ( \Big (\frac{|\mu |}{2} \Big )^{\rho }+ \Big |
\frac{1}{2} \sum _{k=1}^{N} a_{k}
\mu ^{2k+1} (-1)^{k} \Big |^{\rho }
\Big ) \, d|\Gamma _{y} (\lambda , \mu )|
\nonumber
\\
& \leq \int _{-\infty }^{\infty }\int _{-\infty }^{\infty }\frac{1}{2^{2
\rho }}
\Big (|\lambda |^{\rho }+ \Big ( \sum _{k=1}^{N} |a_{k}|
|\lambda
|^{2k+1} \Big )^{\rho }\Big )
\nonumber
\\
& \times \Big (|\mu |^{\rho }+ \Big ( \sum _{k=1}^{N} |a_{k}|
|\mu |^{2k+1} \Big )^{\rho }\Big ) \, d|\Gamma _{y} (\lambda , \mu )|.
\label{CZ_}
\end{align}
This integral converges if the following integral converges
%
\begin{equation}
\label{int_exx2}
\int _{-\infty }^{\infty }\int _{-\infty }^{\infty }|\lambda \mu |^{(2N+1)
\rho }
\, d|\Gamma _{y} (\lambda , \mu )| < \infty .
\end{equation}

Note that for the existence of the solution $U(t, x)$ we have to impose
the condition \eqref{exists_classic}, which for the admissible function\index{admissible function}
$Z(u)=|u|^{\rho }$ takes the form
%
\begin{equation}
\label{exists_classic_1}
\int _{-\infty }^{\infty }\int _{-\infty }^{\infty }\left |\lambda \right |
^{(2N+1)(\rho +1)}
\left |\mu \right |^{(2N+1)(\rho +1)}\, d|\Gamma
_{y} (\lambda , \mu )| <\infty ,
\end{equation}
and implies the fulfilment of \eqref{int_exx2}. Therefore, $C_{Z}^{2}$
is well defined.

If \eqref{exists_classic_1} holds true, then we can define $\sigma (h)$
by means of the formula \eqref{sigma}, and, for our choice of $Z$, we
have that
%
\begin{equation}
\label{sigma_rho}
\sigma (h) = 2 C_{y} C_{Z} h^{\rho }= C h^{\rho },
\quad
0<\rho \le 1,
\end{equation}
where we have denoted $C=2 C_{y} C_{Z}$.

Therefore, in view of Theorem \ref{thm_sigma}, condition
\eqref{cond_sigma} holds with $\sigma (h)$ of the form
\eqref{sigma_rho}, that is, we have the H\"{o}lder continuity of sample
paths of the solution $U(t,x)$.

Let $r(v)=v^{\beta }-1$, $0<\beta <\rho /2$, then $r^{(-1)}(v)=(v+1)^{1/
\beta }$. For such $r(v)$ and the above choice of $Z$ we have
%
\begin{align}
\label{hat_I_r_ex1}
\hat{I}_{r} (\delta )
& = \int _{0}^{\delta }r \left ( \left [ \Big (
\frac{b-a}{2} \Big ( Z^{(-1)}
\Big ( \frac{2 C_{Z} C_{y}}{s} \Big ) - u
_{0} \Big ) +1 \Big ) \right .\right .
\nonumber
\\
& \left .\left .\times \Big ( \frac{d-c}{2} \Big ( Z^{(-1)} \Big ( \frac{2
C_{Z}
C_{y}}{s} \Big ) -u_{0} \Big ) +1\Big )\right ] \right ) \,ds
\nonumber
\\
&\le \int _{0}^{\delta }\left (\left (\frac{\varkappa }{2}\frac{(2 C
_{Z} C_{y})^{1/\rho }}{s^{1/\rho }}+1\right )^{2\beta }-1\right )\,ds.
\end{align}
Consider $\delta \in (0, \theta \varepsilon _{0}]$ and let us choose
$\theta $ such that $ \frac{\varkappa }{2}\big (\frac{2 C_{Z} C_{y}}{
\theta \varepsilon _{0}}\big )^{1/\rho }>1$, that is, $\theta \in (0,
\frac{2 C_{Z} C_{y}}{\varepsilon _{0}}\big (\frac{\varkappa }{2}\big )^{
\rho })$. Then we can write the following estimate:
\begin{eqnarray*}
\hat{I}_{r} (\delta )
&\le & \int _{0}^{\delta }\left (\Big (\varkappa
\Big (\frac{2 C_{Z} C_{y}}{s}\Big )^{1/\rho }\Big )^{2\beta }-1\right )
\,ds
\\
&=&{(2 C_{Z} C_{y})^{2\beta /\rho }\varkappa ^{2\beta }}\left (1-\frac{2
\beta }{\rho }\right )^{-1}\delta ^{1-2\beta /\rho }-\delta .
\end{eqnarray*}
Suppose that $\theta \varepsilon _{0}<\gamma _{0}$, then
\begin{equation*}
r^{(-1)}\left (\frac{\hat{I}_{r}(\min (\theta \varepsilon _{0}, \gamma
_{0}))}{\theta \varepsilon _{0}}\right )\le {(2 C_{Z} C_{y})^{2/\rho }
\varkappa ^{2}}\left (1-\frac{2\beta }{\rho }\right )^{-1/\beta }(
\theta \varepsilon _{0})^{-2/\rho },
\end{equation*}
and
\begin{equation*}
\hat{A}(\theta ,u) \le \exp \Big \{- \frac{u^{\gamma }(1-\theta )^{
\gamma }}{\gamma \varepsilon _{0}^{\gamma }}\Big \}
{(2 C_{Z} C_{y})^{2/
\rho }\varkappa ^{2}}\left (1-\frac{2\beta }{\rho }\right )^{-1/\beta
}(\theta \varepsilon _{0}\xch{)^{-2/\rho }.}{)^{-2/\rho },}
\end{equation*}
If $\beta \to 0$, then $ \bigl(1-\frac{2\beta }{\rho } \bigr)^{-1/
\beta }\to \, e^{2/\rho }$, and we obtain
\begin{equation*}
\hat{A}(\theta ,u) \le \exp \Big \{- \frac{u^{\gamma }(1-\theta )^{
\gamma }}{\gamma \varepsilon _{0}^{\gamma }}\Big \}
{(2 e C_{Z} C_{y})^{2/
\rho }\varkappa ^{2}}(\theta \varepsilon _{0})^{-2/\rho },
\end{equation*}
for $\theta \in (0, \frac{2 C_{Z} C_{y}}{\varepsilon _{0}}\big (\frac{
\varkappa }{2}\big )^{\rho })$, and, therefore, for such $\theta $ we
obtain
%
\begin{equation}
\label{sup_simple}
P \big \{ \sup _{\substack{a \le t \le b
\\
c \le x \le d}}
|U(t, x)| > u \big \} \leq 2\exp \Big \{- \frac{u^{
\gamma }(1-\theta )^{\gamma }}{\gamma \varepsilon _{0}^{\gamma }}
\Big \}
{(2 e C_{Z} C_{y})^{2/\rho }\varkappa ^{2}}(\theta \varepsilon
_{0})^{-2/\rho }.
\end{equation}

Let $ y=\{y(u), u \in \mathbb{R}\}$ be a centered Gaussian random
process. Then $C_{y} = 1$, $ \varphi (x) = \frac{x^{2}}{2}$,
$ \varphi ^{*} (x) = \frac{x^{2}}{2}$.

Then the bound \eqref{sup_simple} takes place with $\gamma =2$, provided
that \eqref{exists_classic_1} holds.

Note that the constant $\rho \in (0, 1]$ should guarantee the
convergence of the integral \eqref{exists_classic_1}, and, therefore,
the choice of $\rho $ in \eqref{sup_simple} depends on the integrability
properties of $\Gamma _{y}$.
\end{ex}

The choice of the function $\varphi $ in Example~\ref{ex1} is reasoned
by the fact that the corresponding class of random process\index{random process} is the
natural generalization of Gaussian processes. This example is rather
simple and, at the same time, it is very illustrative and instructive.
Therefore, the derivations above are worth to be summarized as a
separate statement.

\begin{cor}
\label{thm_sup_solution_corollary}
Let $y=\{y(u),\, u \in \mathbb{R}\}$ be a real strictly
$\varphi $-sub-Gaussian random process with $\varphi (x)=\frac{|x|^{
\alpha }}{\alpha }$, $1 <\alpha \leq 2$, determining constant
$C_{y}$ and $\mathsf{E}y(t) y(s) = \Gamma _{y} (s,t)$. Let $U(t, x) =
\int  _{-\infty }^{\infty }I(t, x, \lambda )\,dy(\lambda )$, where
$I(t, x, \lambda )$ is given by \eqref{Itxlm},\vadjust{\goodbreak} $a \leq t \leq b$,
$c \leq x \leq d$, and $\varepsilon _{0} =
\sup _{\substack{a \le t \le b
\\
c \le x \le d}}\tau _{\varphi }(U(t, x))$. Further, let the constant
$\rho \in (0,1]$ be such that \eqref{exists_classic_1} holds. Then

(i) $U(t,x)$ exists, is continuous with probability one and for its
sample paths the H\"{o}lder continuity holds in the form
\begin{equation*}
\label{cond_sigma??}
\sup _{ \substack{t,t_{1}\in [a,b]:|t- t_{1}| \leq h
\\
x,x_{1}\in [c,d]:|x-x_{1}| \leq h
\\
}}
\tau _{\varphi }( U(t, x) - U(t_{1}, x_{1})) \le 2C_{Z} C_{y} h^{
\rho },
\end{equation*}
where $C_{Z}$ is defined in \eqref{CZsqrt};

(ii) for all $0<\theta <1$ such that $\theta \varepsilon _{0}<2C_{Z} C
_{y} (\varkappa /2)^{\rho }$ with $\varkappa = \max (b-a, d-c)$,
$\gamma $ such that $\frac{1}{\alpha }+\frac{1}{\gamma }=1$, and
$ u > 0$ the following inequality holds true
%
\begin{equation}
\label{P_sup_U_spec}
P \big \{ \sup _{\substack{a \le t \le b
\\
c \le x \le d}}
|U(t, x)| > u \big \} \leq 2\exp \Big \{- \frac{u^{
\gamma }(1-\theta )^{\gamma }}{\gamma \varepsilon _{0}^{\gamma }}
\Big \}
{(2 e C_{Z} C_{y})^{2/\rho }\varkappa ^{2}}(\theta \varepsilon
_{0})^{-2/\rho }.
\end{equation}

In particular, if the process $y=\{y(u),\, u \in \mathbb{R}\}$ is
Gaussian, then
%
\begin{equation}
\label{P_sup_U_gauss}
P \big \{ \sup _{\substack{a \le t \le b
\\
c \le x \le d}}
|U(t, x)| > u \big \} \leq 2\exp \Big \{- \frac{u^{2}
(1-\theta )^{2}}{2\varepsilon _{0}^{2}}\Big \}
{(2 e C_{Z})^{2/\rho }
\varkappa ^{2}}(\theta \varepsilon )^{-2/\rho }
\end{equation}
for all $0<\theta <1$ such that $\theta \Gamma <2C_{Z} (\varkappa /2)^{
\rho }$ and $ u > 0$.

\end{cor}

%
\begin{ex}
\label{other_phi}
Consider $\varphi (x)=\exp \{|x|\}-|x|-1$, $x\in \mathbb{R}$. Then
$\varphi ^{*}(x)=(|x|+1)\ln (|x|+1)-|x|,\,x\in \mathbb{R}$. Hence,
\begin{eqnarray*}
\hat{A}(\theta ,u)
&=& \exp \left \{  - \left (\frac{u (1-\theta )}{
\varepsilon _{0}}+1\right )\ln \left (\frac{u (1-\theta )}{\varepsilon
_{0}}+1\right )+\frac{u (1-\theta )}{\varepsilon _{0}} \right \}
\\
&\times & r^{(-1)}\left (\frac{\hat{I}_{r}(\min (\theta \varepsilon
_{0}, \gamma _{0}))}{\theta \varepsilon _{0}}\right ).
\end{eqnarray*}
Let us take the function $ Z(u) = \ln ^{\alpha }(u+1) $, $ u \geq 0 $,
$ \alpha > 1$, as an admissible\index{admissible function} function for the space $\mathrm{Sub}
_{\varphi }(\Omega )$. In this case,
\begin{align*}
u_{0} &= e^{\alpha } - 1, \qquad Z^{(-1)} (v) = \exp \left \{  v^{\frac{1}{
\alpha }}\right \}  - 1,
\qquad Z(v + u_{0}) = \ln ^{\alpha }(v + e^{
\alpha }),
\\
C_{Z}^{2}& = \int _{-\infty }^{\infty }\int _{-\infty }^{\infty }\Big (
\ln ^{\alpha }\Big (
\frac{|\lambda |}{2} + e^{\alpha } \Big ) +
\ln ^{\alpha }\Big ( \frac{1}{2}
\Big | \sum _{k=1}^{N} a_{k} \lambda ^{2k+
1} (-1)^{k} \Big | + e^{\alpha } \Big ) \Big )
\nonumber
\\&
\times \Big ( \ln ^{\alpha }\Big (
\frac{|\mu |}{2} + e^{\alpha } \Big ) +
\ln ^{\alpha }\Big ( \frac{1}{2}
\Big | \sum _{k=1}^{N} a_{k} \mu ^{2k+ 1}
(-1)^{k} \Big | + e^{\alpha } \Big )
\Big ) \, d|\Gamma _{y} (\lambda ,
\mu )|.
\end{align*}
The above integral converges if the following integral converges
%
\begin{equation}
\label{int_ex1}
\int _{-\infty }^{\infty }\int _{-\infty }^{\infty }\ln ^{\alpha }(|
\lambda | + e^{\alpha })
\ln ^{\alpha }(|\mu | + e^{\alpha }) \, d|
\Gamma _{y} (\lambda , \mu )|.
\end{equation}

That is, if condition (\ref{int_ex1}) holds true, then
Theorem~\ref{thm_sup_solution} holds with
%
\begin{align}
\label{hat_I_r_other}
\hat{I}_{r} (\delta )
& = \int _{0}^{\delta }r \left ( \left [ \Big (
\frac{b-a}{2} \Big ( \exp \left \{
\Big ( \frac{2 C_{Z} C_{y}}{s}
\Big )^{\frac{1}{\alpha }}\right \}  - 1-(e^{\alpha }-1) \Big ) +1
\Big ) \right .\right .
\nonumber
\\
& \left .\left .\times \Big ( \frac{d-c}{2} \Big ( \exp \left \{
\Big ( \frac{2 C_{Z} C_{y}}{s} \Big )^{\frac{1}{\alpha }}\right \}  - 1-(e
^{\alpha }-1) \Big ) +1\Big )\right ] \right ) \,ds
\nonumber
\\
&=\int _{0}^{\delta }r \left ( \left [ \Big ( \frac{b-a}{2} \Big (
\exp \left \{
\Big ( \frac{2 C_{Z} C_{y}}{s}
\Big )^{\frac{1}{\alpha }}\right \}  - e^{\alpha }\Big ) +1 \Big ) \right .\right .
\nonumber
\\
& \left .\left .\times \Big ( \frac{d-c}{2} \Big ( \exp \left \{
\Big ( \frac{2 C_{Z} C_{y}}{s} \Big )^{\frac{1}{\alpha }}\right \}  - e
^{\alpha }\xch{\Big )}{) \Big )} +1\Big )\right ] \right ) \,ds.
\end{align}
Let $ \frac{d-c}{2}\, e^{\alpha } > 1$ and $ \frac{b -a}{2} e^{\alpha
} > 1$.

That is, we choose some $\alpha >\max  \bigl\{  1,\ln \bigl (\frac{2}{b-a}\bigr),
\ln \bigl(\frac{2}{d-c}\bigr)\bigr\}  $. Then
%
\begin{align}
\label{hat_I_r_other_estimate1}
\hat{I}_{r} (\delta )
& \leq \int _{0}^{\delta }r \left (
\frac{(b-a)(d-c)}{4} \exp \left \{
2\Big ( \frac{2 C_{Z} C_{y}}{s}
\Big )^{\frac{1}{\alpha }}\right \}  \right ) \,ds
\nonumber
\\
& \leq \int _{0}^{\delta }r \left ( \frac{\varkappa ^{2}}{4} \exp
\left \{
2\Big ( \frac{2 C_{Z} C_{y}}{s}
\Big )^{\frac{1}{\alpha }}\right \}  \right ) \,ds.
\end{align}
In our case, the function $r(v)=\ln v$, $v\ge 1$, satisfies the
conditions defined in Theorem~\ref{Theorem5_KO2016} and is convenient
for the estimation of $\hat{I}_{r} (\delta )$. Substituting it in the
expression above, we get
%
\begin{align}
\label{hat_I_r_other_estimate2}
\hat{I}_{r} (\delta )
& \leq \int _{0}^{\delta }\ln \left ( \frac{
\varkappa ^{2}}{4} \exp \left \{
2\Big ( \frac{2 C_{Z} C_{y}}{s}
\Big )^{\frac{1}{\alpha }}\right \}  \right ) \,ds
\nonumber
\\
&=\delta \ln \left (\frac{\varkappa ^{2}}{4}\right )+\int _{0}^{\delta
}2\Big ( \frac{2 C_{Z} C_{y}}{s} \Big )^{\frac{1}{\alpha }} \,ds
\nonumber
\\
&= \delta \ln \left (\frac{\varkappa ^{2}}{4}\right )+\frac{2({2 C
_{Z} C_{y}})^{\frac{1}{\alpha }}\delta ^{1-\frac{1}{\alpha }}}{1-\frac{1}{
\alpha }}.
\end{align}
Since $r^{(-1)}(v)=e^{v},\,v\ge 0$, for such $\theta \in (0,1)$
that $\theta \varepsilon _{0}<\gamma _{0}$ we obtain
\begin{equation*}
r^{(-1)}\left (\frac{\hat{I}_{r}(\min (\theta \varepsilon _{0}, \gamma
_{0}))}{\theta \varepsilon _{0}}\right )\le \exp \left \{  \ln \left (\frac{
\varkappa ^{2}}{4}\right )+\frac{2\alpha }{\alpha -1}\left (\frac{2 C
_{Z}
C_{y}}{\theta \varepsilon _{0}}\right )^{\frac{1}{\alpha }}\right \}
\end{equation*}
and, finally, for all $ u > 0$ the following inequality emerges
\begin{eqnarray*}
\label{P_sup_U_ex2}
P \big \{ \sup _{\substack{a \le t \le b
\\
c \le x \le d}}
|U(t, x)| > u \big \}
&\leq & 2\hat{A}(\theta ,u)
\\
&\leq & 2\exp \left \{  - \left (\frac{u (1-\theta )}{\varepsilon _{0}}+1\right )
\ln \left (\frac{u (1-\theta )}{\varepsilon _{0}}+1\right )\right .
\\
&+&\left . \frac{u (1-\theta )}{\varepsilon _{0}} + \ln \left (\frac{
\varkappa ^{2}}{4}\right )+\frac{2\alpha }{\alpha -1}\left (\frac{2 C
_{Z}
C_{y}}{\theta \varepsilon _{0}}\right )^{\frac{1}{\alpha }}\right \}  .
\end{eqnarray*}

\end{ex}

\begin{appendix}
\section*{Appendix A. Proof of Theorem \ref{thm_sigma}}

The derivation of the bound \eqref{cond_sigma} for the process
$U(t, x)$ is based on the particular structure of this process and on
the use of the property \eqref{**} of admissible function~$Z$.

Firstly note that the process $ U(t, x)$ is separable since
$ U(t, x)$ is continuous with probability one. The process $U(t, x)$ is
strictly $\varphi $-sub-Gaussian with the determining constant
$C_{y}$, and, therefore, we can write:
\begin{equation*}
\sup _{ \substack{|t- t_{1}| \leq h
\\
|x-x_{1}| \leq h
\\
}}
\tau _{\varphi }( U(t, x) - U(t_{1}, x_{1})) \leq C_{y}
\sup _{ \substack{|t- t_{1}| \leq h
\\
|x-x_{1}| \leq h
\\
}}\left (\mathsf{E}(U(t, x) - U (t_{1}, x_{1}))^{2} \right )^{1/2}.
\end{equation*}
We can also estimate $\varepsilon _{0} =\sup  _{(t, x)\in D}
\tau _{\varphi }(U(t, x)) $, where $D=\{a \leq t \leq b$, $c \leq x
\leq d\}$, as follows:
%
\begin{align}
\varepsilon _{0}
& \leq C_{y} \sup \limits _{(t, x)\in D}
\left (\mathsf{E}|U(t, x)|^{2}\right )^{\frac{1}{2}}
\nonumber
\\
& \le C_{y} \sup \limits _{(t, x)\in D}\left (\int _{-\infty }^{\infty
}\int _{-\infty }^{\infty }\big | I(t, x, \lambda )
I(t, x, \mu )
\big | \, d|\Gamma _{y} (\lambda , \mu )|\right )^{\frac{1}{2}}
\nonumber
\\
& \leq C_{y} \left (\int _{-\infty }^{\infty }\int _{-\infty }^{\infty
}\,
d|\Gamma _{y} (\lambda , \mu )|\right )^{\frac{1}{2}} =: \Gamma .
\tag{A.1}\label{A.1}
\end{align}

It is enough to consider one of the representations for initial
condition, $\eta (x)=\int _{-\infty }^{\infty } \sin xu \,\mathrm{d}y(u)$
or $\eta (x)=\int _{-\infty }^{\infty } \cos xu \,\mathrm{d}y(u)$, since
the proofs are similar for both cases.

Let us consider $\mathsf{E}(U(t,x) -U(t_{1}, x_{1}))^{2}$ for
$\kappa (u) = \cos (u)$:
%
\begin{align}
\mathsf{E}(U(t, x) - U(t_{1}, x_{1}))^{2} = \left ( \int _{-\infty }
^{\infty }(I(t, x, \lambda ) -
I(t_{1}, x_{1} , \lambda ))\, dy(
\lambda )\right )^{2}
\nonumber
\\
\leq \int _{-\infty }^{\infty }\int _{-\infty }^{\infty }|I(t, x,
\lambda ) -
I(t_{1}, x_{1} , \lambda )|
|I(t, x, \mu ) - I(t_{1}, x
_{1} , \mu )| \, d |\Gamma _{y} (\lambda , \mu )|.\nonumber
\end{align}
We can write $|I(t, x, \lambda ) - I(t_{1}, x_{1} , \lambda )| = |
\cos A - \cos B|$, where
\begin{equation*}
A = x\lambda + t \sum _{k=1}^{N} a_{k} \lambda ^{2k+1} (-1)^{k}, \qquad
B = x_{1}\lambda + t_{1} \sum _{k=1}^{N} a_{k} \lambda ^{2k+1} (-1)^{k}.
\end{equation*}
Thus
\begin{eqnarray*}
|I(t, x, \lambda ) - I(t_{1}, x_{1}, \lambda )|
& =& 2 \Big |
\sin \frac{A+ B}{2}
\sin \frac{B- A}{2} \Big |
\\
&\leq & 2 \Big | \sin \frac{B-A}{2}\Big |
= 2 \Big |\sin (C+D)\Big |,
\end{eqnarray*}
where
\begin{equation*}
C = \frac{\lambda (x_{1} - x)}{2}, \qquad D = \frac{t_{1} - t}{2}
\sum _{k=1}^{N} a_{k} \lambda ^{2k+1} (-1)^{k}.
\end{equation*}
Therefore,
\begin{align*}
2 | \sin (C+D)|
& = 2| \sin C \cos D + \cos C \sin D|
\leq 2( |\sin C|
+ |\sin D|)
\\
& = 2 \Big ( \Big | \sin \frac{\lambda (x_{1} -x)}{2} \Big | + \Big |
\sin \frac{(t_{1} - t)}{2} \sum _{k=1}^{N} a_{k} \lambda ^{2k+1} (-1)^{k}
\Big | \Big ).
\end{align*}
Choose some admissible function $\{Z( x),x\ge 0\}$ for the space
$\mathrm{Sub}_{\varphi }(\Omega )$. In view of Lemma~\ref{lmm12}, we can
write
\begin{align*}
| I(t,x, \lambda ) - I( t_{1}, x_{1}, \lambda )|
&\leq 2 Z^{-1}
\Big ( \frac{1}{|x - x_{1}| }+ u_{0}\Big ) Z \Big (
\frac{|\lambda | }{2}+u_{0} \Big )
\\
&+ 2 \, Z^{-1} \Big (\frac{1}{|t- t_{1}|} +u_{0} \Big ) Z \Big (
\frac{1}{2} \Big | \sum _{k=1}^{N} a_{k} \lambda ^{2k+1} (-1)^{k} \Big | +
u_{0} \Big ).
\end{align*}
Thus, we obtain:
%
\begin{align}
& \sup _{ \substack{|t- t_{1}| \leq h \\ |x-x_{1}| \leq h \\}}
\tau _{\varphi }( U(t, x) - U(t_{1}, x_{1})) \leq C_{y}
\sup _{ \substack{|t- t_{1}| \leq h \\ |x-x_{1}| \leq h \\}}
\left (\mathsf{E}(U(t, x) - U (t_{1}, x_{1}))^{2} \right )^{1/2}
\nonumber
\\
& \leq C_{y} \frac{2}{Z \big ( \frac{1}{h} + u_{0} \big )} \left [
\int \limits _{-\infty }^{\infty }\int \limits _{-\infty }^{\infty }\left (
Z
\Big ( \frac{|\lambda |}{2} + u_{0} \Big ) + Z\Big ( \frac{1}{2}
\Big |
\sum _{k=1}^{N} a_{k} \lambda ^{2k+1} (-1)^{k} \Big | + u_{0}
\Big )
\right )\right .
\nonumber
\\
& \times \left . \left ( Z \big (
\frac{|\mu |}{2} +u_{0}\big ) + Z
\Big ( \frac{1}{2} \Big | \sum _{k=1}^{N}
a_{k} \mu ^{2k+1} (-1)^{k}
\Big | + u_{0} \Big ) \right )\, d|\Gamma _{y}
(\lambda , \mu )|\right ]
^{1/2}
\nonumber
\\
& = C_{y} C_{Z} \frac{2}{Z \big (\frac{1}{h} +u_{0} \big )}.\nonumber
\label{(5.9)}
\end{align}
For $\kappa (u) = \sin u$ we have the same inequality. So, in the
notations of Theorem~\ref{sup_X}
\begin{equation*}
\sigma (h ) = 2 C_{y} C_{Z} \Big ( Z\Big ( \frac{1}{h} +u_{0} \Big )
\Big )^{-1},
\quad 0<h<\max (b-a,d-c).
\end{equation*}
\end{appendix}

\begin{acknowledgement}
We are grateful to editors and reviewers for valuable remarks and
suggestions, which helped us to improve the paper significantly.
\end{acknowledgement}


\end{document}